\documentclass{amsart}

\usepackage{amsmath,amssymb,amsthm}

\newtheorem{prop}{Proposition}[section]
\newtheorem{thm}[prop]{Theorem}

\newtheorem{cor}[prop]{Corollary}

\theoremstyle{definition}

\newtheorem{rem}[prop]{Remark}

\newtheorem*{ack}{Acknowledgement}

\def\co{\colon\thinspace}

\newcommand{\C}{\mathbb C}

\newcommand{\rmd}{\mathrm d}

\newcommand{\rme}{\mathrm e}

\newcommand{\rmi}{\mathrm i}

\newcommand{\N}{\mathbb N}

\newcommand{\R}{\mathbb R}

\newcommand{\Z}{\mathbb Z}

\DeclareMathOperator{\eh}{\mathrm{EH}}

\DeclareMathOperator{\length}{\mathrm{length}}

\begin{document}

\author[Kai Zehmisch]{Kai Zehmisch}
\address{Mathematisches Institut, Universit\"at zu K\"oln, Weyertal 86--90, 50931 K\"oln,Germany}
\email{kai.zehmisch@math.uni-koeln.de}

\title[Lagrangian non-squeezing]{Lagrangian non-squeezing and a geometric inequality}

\date{28.10.2013}

\begin{abstract}
  We prove that if the unit codisc bundle of a closed Riemannian manifold
  embeds symplectically into a symplectic cylinder of radius one
  then the length of the shortest non-trivial closed geodesic is at most
  half the area of the unit disc.
\end{abstract}

\thanks{The author is partially supported by DFG grant
ZE 992/1-1.}

\maketitle


\section{Introduction\label{intro}}

Consider a metric $g$ on the unit circle $S^1=\partial D$.
After a reparametrization by arc length $g=r^2g_0$
becomes a positive multiple of the standard metric on $S^1$.
If the unit codisc bundle of $(S^1,g)$
embeds into the unit disc $D$
such that the area and the orientation
are preserved
then
\[
2\length_g(S^1)\leq\pi.
\]
This is equivalent to $r\leq\tfrac14$.
In other words
\[
\inf(g)\leq\frac{\pi}{2},
\]
where $\inf(g)$ denotes the length
of the shortest non-trivial closed geodesic.
In this article
we prove a symplectic generalization of this inequality.
For any closed Riemannian manifold $(L,g)$
such that the unit codisc bundle $D^*(g)L$
has a symplectic embedding into the cylinder
$Z=D\times\R^{2n-2}$
which is provided with the standard split symplectic form
$\inf(g)\leq\frac{\pi}{2}$ holds true.
In fact,
if the dimension is $\geq4$
we only require that the symplectic embedding
exists in a neighbourhood 
of the unit cosphere bundle $S^*(g)L$,
see Corollary \ref{satz1}.

The proof of this result
is based on the {\bf link capacity} $\ell$
that we introduce in Theorem \ref{tool}.
For a given subset $U$
the capacity $\ell$
measures the largest minimal total action
of null-homologous Reeb links
on a contact type hypersurface in $U$
that is diffeomorphic to a unit cotagent bundle.
This is a variant
of the capacity introduced in \cite{geizeh11,geizeh12},
cf.\ \cite{zehzil12}
and fits into a larger class of
so called \emph{embedding capacities},
see \cite{chls07}.

The first embedding capacity appeared
in the unpublished work \cite{cielmoh},
cf.\ \cite{chls07}.
In \cite{cielmoh} Cieliebak and Mohnke
defined a Lagrangian embedding capacity for $2$-connected
symplectic manifolds $(V,\omega)$
\[
c_L(V,\omega)
:=\sup\{\inf(L)\,|\,
\text{$L\subset(V,\omega)$}\},
\]
where $\inf(L)$ denotes the least positive symplectic area
of a smooth disc in $V$ with boundary on $L$.
Here the supremum runs over all Lagrangian
\emph{tori} in $(V,\omega)$.
The values on the unit symplectic cylinder $Z$,
unit ball $B$, and unit polydisc $P$
are
\[
c_L(Z)=\pi,\qquad
c_L(B)=\frac{\pi}{n},\qquad\text{and}\qquad
c_L(P)=\pi.
\]
This capacity gives an alternative proof
of the Ekeland-Hofer non-squeezing theorem \cite[Corollary 3]{ekho90}.
It states
that if the polydisc
\[
P(r_1,\ldots,r_n):=D_{r_1}\times\ldots\times D_{r_n}
\]
with radii $0<r_1\leq\ldots\leq r_n$
embeds symplectically into the ball $B_R$ of radius $R$
then $\sqrt{n}\;r_1\leq R$.
Generalizing to arbitrary Lagrangian submanifolds
Swoboda and Ziltener \cite{swozil12,swozil12b}
obtained a symplectic capacity $a_L\geq c_L$
for $2$-connected symplectic manifolds $(V,\omega)$ via
\[
a_L(V,\omega)
:=\sup\{\inf(L)\,|\,
\text{$L\subset(V,\omega)$ closed Lagrangian submanifold}\}
\]
such that
\[
a_L(Z)=\pi,\qquad
a_L(B)\geq\frac{\pi}{2}.
\]
The precise value on the unit ball is not known.
In fact Swoboda and Ziltener assign
a capacity to a large class of coisotropic submanifolds
for every possible codimension
and prove non-squeezing results for so-called small sets,
see \cite{swozil12,swozil12b}.
The link capacity (see Theorem \ref{tool})
can be seen as another example in this direction.

In Corollary \ref{nsq1} and \ref{nsq2}
we give further non-squeezing results for Lagrangian submanifolds.
For that we use 
Chekanov's elementary tori,
see \cite{chek96},
Damian's proof of the Audin conjecture in the monotone case,
see \cite{damian},
and Lagrangian embedding capacities
that we introduce in Section \ref{area}.


\section{Measuring the area\label{area}}

We are interested in {\bf special capacities} $a$
on the standard symplectic vector space $\R^{2n}$,
which are (1) monotone on subsets of $\R^{2n}$,
i.e.\ $a(U_1)\leq a(U_2)$ provided $U_1\subset U_2$,
(2) invariant under \emph{global}
symplectomorphisms of $\R^{2n}$,
(3) conformal in the sense that $a(rU)=r^2a(U)$
for all $U\subset\R^{2n}$ and $r\in\R$,
and (4) satisfy
\[
a(Z)<\infty,\qquad
a(B)>0,
\]
see \cite[p.\ 172]{hoze94}.
The aim is to measure the minimal symplectic area
$\inf(L)$ of closed Lagrangian submanifolds
$L$ in $\R^{2n}$
among \emph{all} smooth discs attached to $L$.
In other words we consider
the Liouville class $\lambda_L=[\lambda |_{TL}]$
for any primitive $\lambda$ of
$\rmd\mathbf{x}\wedge\rmd\mathbf{y}$.
The image of $H_1(L;\Z)$ under $\lambda_L$
generates a subgroup $\Lambda_L$ of $\R$.
If this group is discrete we call $L$ {\bf rational}
and $\inf(L)$ is the positive generator of $\Lambda_L$;
otherwise $\inf(L)$ is zero, see \cite{polt01}.

For our first version of a special capacity
we consider for real numbers 
\[
0<r_1\leq\ldots\leq r_n
\]
the
{\bf elementary Lagrangian tori}
\[
T(r_1,\ldots,r_n):=
\partial D_{r_1}\times\ldots\times\partial D_{r_n}
\]
in $\R^{2n}$.
We call two closed Lagrangian submanifolds of $\R^{2n}$
{\bf symplectomorphic}
if there exists a global symplectomorphism of $\R^{2n}$
which maps one to the other.
It follows from
\cite[Theorem A]{chek96}
that the first radius
\[
r_1=r_1(L)
\]
of a Lagrangian torus $L$ symplectomorphic
to $T(r_1,\ldots,r_n)$
is an invariant under global symplectomorphisms.

\begin{thm}
  For subsets $U$ in $\R^{2n}$ the quantity
  \[
  a_e(U)
  :=\sup\big\{\pi \big(r_1(L)\big)^2\,|\,
  \text{$L\subset U$}\big\},
  \]
  where the supremum is taken over
  all Lagrangian tori $L$ symplectomorphic to an
  elementary torus,
  defines a special capacity in $\R^{2n}$
  such that
  \[
  a_e(Z)=\pi,\qquad
  a_e(B)=\frac{\pi}{n}.
  \]
\end{thm}

\begin{proof}
  We only have to verify the normalization axiom.
  For the lower bounds consider
  the tori $T_1$ and $T_{1/\sqrt{n}}$
  that have minimal symplectic action
  $\pi$ and $\pi/n$, resp.
  To obtain upper bounds consider
  a Lagrangian torus $L$ in $\R^{2n}$
  which is symplectomorphic
  to the elementary torus $T(r_1,\ldots,r_n)$.
  For $r_1=r_1(L)$ the values of the
  first and $n$-th Ekeland-Hofer capacity of $L$
  are
  \[
  c_1^{\eh}(L)=\pi r_1^2,\qquad
  c_n^{\eh}(L)=n\pi r_1^2,
  \]
  see \cite[Theorem 2.1]{chek96}.
  The claim follows now from
  \[
  c_1^{\eh}(Z)=\pi=c_n^{\eh}(B)
  \]
  and the monotonicity property
  of the Ekeland-Hofer capacities, see \cite{ekho90}.
\end{proof}

\begin{rem}
  Because $c_1^{\eh}$ takes the value $\pi r_1^2$
  on the polydisc $P(r_1,\ldots,r_n)$ 
  the proof shows that
  $a_e\big(P(r_1,\ldots,r_n)\big)=\pi r_1^2$.
\end{rem}

A direct consequence of the theorem is
that the torus $T(r_1,\ldots,r_n)$
admits a global symplectic embedding
into the symplectic cylinder $Z_R$ of radius $R$
if and only if $r_1\leq R$.
This non-squeezing result follows
alternatively from the stronger
\cite[Main Theorem]{chek98},
which gives an upper bound
on the area of a non-constant holomorphic disc
(for example for the standard complex structure)
attached to $T(r_1,\ldots,r_n)$
by its displacement energy.
The rational case
was observed by Sikorav in \cite{sik91}.
Note that Sikorav's theorem implies
the general case by approximating
irrational radii by rational numbers.

\begin{cor}
  \label{nsq1}
  If the torus $T(r_1,\ldots,r_n)$
  admits a global symplectic embedding
  into the ball $B_R$ of radius $R$
  then $\sqrt{n}\;r_1\leq R$.
\end{cor}

\begin{rem}
  This follows alternatively
  with the Cieliebak-Mohnke capacity,
  see \cite{cielmoh},
  via an approximation by rational Lagrangian tori.
\end{rem}

A second special capacity on $\R^{2n}$
can be constructed as follows:
Consider closed connected monotone
Lagrangian submanifolds $L\subset\R^{2n}$
which admit a metric of
non-positive sectional curvature
(and are therefore aspherical by the Hadamard-Cartan Theorem).
Notice that $L$ is allowed to be non-orientable
so that for example in dimension $4$
the curvature condition is not a restriction,
see \cite[0.4.A$_2$]{grom85}.

\begin{thm}
  For subsets $U$ in $\R^{2n}$
  the quantity
  \[
  a_m(U)
  :=\sup\{\inf(L)\,|\,
  \text{$L\subset U$}\},
  \]
  where the supremum
  is taken over all closed connected monotone
  Lagrangian submanifolds $L\subset\R^{2n}$
  which admit a metric of
  non-positive sectional curvature,
  defines a special capacity in $\R^{2n}$
  such that
  \[
  a_m(Z)=\pi,\qquad
  a_m(B)=\frac{\pi}{n}.
  \]
\end{thm}

\begin{proof}
  We only have to show that 
  $a_m(Z)=\pi$ and $a_m(B)=\pi/n$.
  The tori $T_1$ and $T_{1/\sqrt{n}}$
  yield lower bounds.
  Uniform upper bounds are obtained as follows:
  Consider a Lagrangian submanifold
  $L\subset\R^{2n}$ as above.
  Because $L$ is monotone
  the Liouville class $\lambda_L$
  and the Maslov class $\mu_L$
  are related by
  \[
  \lambda_L=\eta\mu_L
  \]
  for some $\eta>0$.
  By Damian's proof
  of the Audin conjecture
  we find a closed curve $\gamma$ on $L$
  such that $\mu_L(\gamma)\leq2$
  (equality if and only if $L$ is orientable),
  see \cite[Theorem 1.5.(a)]{damian}.
  This gives
  \[
  \inf(L)\leq\lambda_L(\gamma)\leq2\eta.
  \]
  Moreover, by Bates \cite[Theorem 3]{bates98},
  the $k$-th Ekeland-Hofer capacity satisfies
  \[
  2k\eta\leq c_k^{\eh}(L),
  \]
  where the curvature condition is used,
  so that
  \[
  \inf(L)\leq\frac{c_k^{\eh}(L)}{k}.
  \]
  The claim follows now from the
  properties of the first and $n$-th
  Ekeland-Hofer capacity.
\end{proof}

\begin{cor}
  \label{nsq2}
  Let $L\subset B_R$ be a closed
  connected Lagrangian submanifold.
  Then
  \[
  \inf(L)\leq\frac{\pi}{n}R^2
  \]
  provided $L$ is monotone
  and admits a metric of non-positive sectional curvature.
\end{cor}

\begin{rem}
  The case without the monotonicity assumption
  follows from the Cieliebak-Mohnke capacity \cite{cielmoh}.
\end{rem}


\section{Measuring the length\label{length}}

For irrational Lagrangian submanifolds
the symplectic area can be arbitrary small,
thus do not give a sensible invariant.
But the length of closed unit speed geodesics on $L$
for certain Riemannian metrics
is an alternative way
to measure the size of Lagrangian submanifolds $L$
symplectically.
\cite{geizeh11,geizeh12} construct the so-called
{\bf orbit capacity}
\[
c(V,\omega)=\sup\{{\inf}_{\ell}(\alpha)\,|\,
\text{$\exists$ contact type embedding
  $(M,\alpha)\hookrightarrow (V,\omega)$}\}
\]
for all symplectic manifolds
$(V,\omega)$ with dimension $\geq4$.
Here ${\inf}_{\ell}(\alpha)$
is the infimum of the total action
of null-homologous Reeb links
on the closed contact manifold $(M,\alpha)$.
As the arguments in \cite{geizeh12} show
closed Reeb orbits that constitute a
null-homologous one-component Reeb link
can be assumed to be contractible.
The supremum is taken over all
embeddings $j\co M\hookrightarrow V$
such that near $j(M)$ there
is a Liouville vector field $Y$ for $\omega$
satisfying $\alpha=j^*(i_Y\omega)$.
Notice that this is equivalent to $d\alpha=j^*\omega$
for the contact form $\alpha$,
see \cite[p.\ 119]{hoze94}.

If one restricts
in the definition of the orbit capacity $c$
to manifolds $M$ that are
diffeomorphic to the unit cosphere bundle $S^*Q$
of closed Riemannian manifolds $Q$
one obtains the link capacity $\ell$.

\begin{thm}
  \label{tool}
  For symplectic manifolds $(V,\omega)$
  with dimension $\geq4$ the quantity
  $\ell(V,\omega)$
  is an intrinsic capacity
  such that
  \[
  \ell(Z)=\pi,\qquad
  \ell(B)\geq\frac{\pi}{n}.
  \]
\end{thm}

\begin{proof}
  Because of $\ell\leq c$
  we only need to compute
  the values on the ball and the cylinder.
  Identify the cotangent bundle
  of the unit circle $S^1=\partial D$
  with $(\R\times S^1,s\rmd t)$.
  Consider polar coordinates on $\C$
  such that the radial Liouville primitive
  of the standard symplectic form
  equals $\tfrac12 r^2\rmd\theta$.
  For $a>0$ we define a symplectic embedding
  \[
  \varphi_a(s,t)=\sqrt{a+2s}\;\rme ^{\rmi t}
  \]
  of $\{s>-\tfrac{a}{2}\}$ into $\C$.
  The image of the $b$-codisc bundle 
  \[
  D_b^* S^1=(-b,b)\times S^1\;,
  \]
  $b\in(0,\tfrac{a}{2})$,
  is the annulus
  \[
  A(a,b)=
  A_{\sqrt{a+2b},\sqrt{a-2b}}=
  D_{\sqrt{a+2b}}\setminus\overline{D}_{\sqrt{a-2b}}\;.
  \]
  For real numbers
  $a_1,\ldots,a_n,b_1,\ldots,b_n$
  with $0<b_j<\tfrac{a_j}{2}$ for $j=1,\ldots,n$
  the embedding
  \[
  \varphi_{a_1}\times\ldots\times\varphi_{a_n}
  \]
  maps $D_{b_1}^* S^1\times\ldots\times D_{b_n}^* S^1$
  onto the polyannulus $A(a_1,b_1)\times\ldots\times A(a_n,b_n)$
  symplectically.
  
  In order to compute the quantity $\ell$
  we consider the flat torus $T^n=\R^n/2\pi\Z^n$.
  The $b$-cosphere bundle $S_b^* T^n$
  is provided with the canonical contact form.
  Each closed geodesic
  induces a null-homologous Reeb link
  with two components
  corresponding to the opposing
  orientations of the geodesic.
  Hence,
  the smallest total action $\inf_{\ell}$
  equals $4\pi b$,
  see \cite[Section 1.5]{geig08}. 
  Because the $b$-codisc bundle $D_b^* T^n$
  is contained in $(D_b^* S^1)^n$
  the images of the $b'$-cosphere bundles
  under $(\varphi_a)^n$ for $b'<b$
  are hypersurfaces of contact type.
  Taking the limits
  $a\!\downarrow\!\tfrac{1}{2}$ and
  $b\!\uparrow\!\tfrac{1}{4}$,
  resp., $a\!\downarrow\!\tfrac{1}{2n}$ and
  $b\!\uparrow\!\tfrac{1}{4n}$
  proves the claim.
\end{proof}

\begin{rem}
  For $\varepsilon>0$ sufficiently small
  the disc of radius $2\sqrt{b}-\varepsilon$
  embeds into the square $(-b,b)\times (0,2\pi)$
  preserving the orientation and the area,
  cf.\ \cite[p.\ 171]{hoze94}.
  Composing this with
  $\varphi_{a_1}\times\ldots\times\varphi_{a_n}$
  appropriately
  yields an symplectic embedding of the polydisc
  \[
  P\big(2\sqrt{b_1}-\varepsilon,\ldots,2\sqrt{b_n}-\varepsilon\big)
  \]
  into the polyannulus
  \[
  A(a_1,b_1)\times\ldots\times A(a_n,b_n)\subset
  P\big(\sqrt{a_1+2b_1},\ldots,\sqrt{a_n+2b_n}\big).
  \]
  Therefore,
  one shows that $\ell\big(P(r_1,\ldots,r_n)\big)=\pi r_1^2$
  as in the proof above.
  Moreover,
  if we consider the metric on $T^n$ induced from $\R^n$
  we see together with the remark after
  \cite[Theorem 4.5]{geizeh12}
  that
  \[
  \ell\big(D^*_{b_1}T^n\big)=
  4\pi b_1=
  \ell\big(D^*_{b_1}S^1\times\ldots\times D^*_{b_n}S^1\big),
  \]
  where we assume that $0<b_1\leq\ldots\leq b_n$.
\end{rem}

Consider a closed monotone Lagrangian submanifold
$L\subset\R^{2n}$ that admits a metric $g$
of non-positive sectional curvature.
By Weinstein's neighbourhood theorem \cite{wein71}
there exists $r>0$
such that the $r$-codisc bundle of $L$ embeds symplectically.
We denote its image by $U_r\subset\R^{2n}$.
Then \cite[Theorem 2.1]{bates98}
implies that for all $k\in\N$
\[
\inf(L)+r\inf(g)\leq c_k^{\eh}(U_r),
\]
where $\inf(g)$ denotes the length
of the shortest non-trivial closed geodesic of $(L,g)$.
In particular,
if $U_r\subset Z_R$ then $r\inf(g)\leq\pi R^2$.
This observation generalizes to the following non-squeezing result.

\begin{cor}
  \label{satz1}
  Let $(Q,g)$ be a closed Riemannian manifold
  such that all closed geodesics are not contractible.
  If a neighbourhood of the $r$-cosphere bundle
  in $T^*Q$ embeds into $Z_R$
  symplectically then
  \[
  2r\inf(g)\leq\pi R^2.
  \]
\end{cor}

\begin{proof}
  The claim is 
  an application of the capacity $\ell$.
  Notice that $2r\inf(g)\leq\inf_{\ell}(\alpha_r)$
  if computed with respect to the restriction $\alpha_r$
  of the canonical Liouville form to $TS^*_r(g)Q$.
\end{proof}

For a continuative discussion the reader is referred to \cite[Section 3.6]{zeh13}.


\begin{ack}
  I thank
  Hansj\"org Geiges, Janko Latschev, and Stefan Suhr
  for their comments on the first version of these notes,
  Felix Schlenk for directing my attention to the work
  of S. M. Bates \cite{bates98}
  during his stay at the Universit\"at Leipzig.
  Further I would like to thank Fabian Ziltener
  for drawing my interest towards this problem in the first place.
  Part of the research in this article was carried out during the conference
  \emph{Periodic Orbits In Contact and Riemannian Geometry}
  from September 3rd through September 7th 2012 in Le Touquet-Paris-Plage.
  I would like to thank the organizers
  Juan-Carlos \'Alvarez Paiva and Florent Balacheff
  as well as
  Urs Frauenfelder,
  Emmanuel Opshtein,
  Yaron Ostrover,
  Federica Pasquotto, and
  Ana Rechtman
  for many stimulating discussions.
\end{ack}


\end{document}